\documentclass[12pt,a4paper]{amsart}
\usepackage{enumerate}
\allowdisplaybreaks
\usepackage{color}

\usepackage{amsmath,amssymb,latexsym}
\def\en{\mathbb{N}}
\def\zet{\mathbb{Z}}
\def\ha{\mathbb H}

\newcommand{\twoline}[2]{\genfrac{}{}{0pt}{}{#1}{#2}}
\newtheorem{theorem}{Theorem}
\newtheorem{lemma}[theorem]{Lemma}
\newtheorem*{lemma*}{Lemma}
\newtheorem{corollary}[theorem]{Corollary}
\newtheorem*{corollary*}{Corollary}

\newtheorem*{definition*}{Definition}
\newtheorem*{theorem*}{Theorem}
\author[M. Paluszynski]{Maciej Paluszy\'nski}
\address{\small Instytut Matematyczny, Uniwersytet Wroc\l awski,
 pl. Grunwaldzki 2/4, \hbox{50-384} Wroc\-\l aw, Poland}
\email{mpal@math.uni.wroc.pl}

\author[J. Zienkiewicz]{Jacek Zienkiewicz}
\address{\small
 Instytut Matematyczny, Uniwersytet Wroc\l awski,
 pl. Grunwaldzki 2/4, \hbox{50-384} Wroc\-\l aw, Poland}
\email{zenek@math.uni.wroc.pl}

\title[ ]{On  maximal function of discrete rough truncated Hilbert transforms}
\begin{document}
\begin{abstract} We prove the weak type (1,1) estimate for maximal function of the truncated rough Hilbert transform
considered in \cite{PZ1}, \cite{PZ2}.
\end{abstract}

\date{\today}
\thanks{The second named author was supported by a NCN grant}
 \maketitle
\baselineskip=18pt
\section{Introduction and Statement of the Results.}

Let $N$ be a dyadic integer, $1<\alpha\le 1+\frac1{1000}$ and let
\begin{equation}
\mu_N(x)=\sum_{N\le m\le 2N}\phi\Big(\frac mN\Big)\,\frac{\delta_0(x-[m^\alpha])}{m}
\end{equation}
where $\delta_0$ denotes the unit
point mass at zero, and $\phi\in C_c^\infty(1,2)$. For fixed positive $\theta<1$ and $\check\mu_N(x)=\mu_N(-x)$ we define
\begin{equation}\label{main_object_1}
H_Mf(x)=\sum_{M^\theta\le N\le M}(\mu_N-\check\mu_N)*f(x).
\end{equation}

The operators $H_M$ can be viewed as a discrete analog of rough singular integral operators, see eg. \cite{C1}. They have appeared in the $\ell^{1,\infty}$ invertibility
problem for discrete singular integral operators considered by the authors
in \cite{PZ1}, \cite{PZ2}.
In particular, it has been proved in \cite{PZ2} that $H_M$ is of weak type (1,1), uniformly
with $M$.

In the current paper we consider the following maximal variant of $H_M$
\begin{equation}\label{main_object}
H_M^*f(x)=\max_{B\le M}\Big|\sum_{M^\theta\le N\le B}(\mu_N-\check\mu_N)*f(x)\Big|.
\end{equation}
It is by now classical, that operators $H_M$ and $H_M^*$ are bounded on $\ell^p$, $p>1$. However, their behavior on $\ell^1$
still seems to be an open question. In this direction
we prove the following partial result.
\begin{theorem}\label{theorem}
The maximal Hilbert  operators $H_M^*f$  defined by \eqref{main_object}   are of weak type (1,1), uniformly in $M$.
\end{theorem}
\begin{corollary}\label{corollary}  The global maximal Hilbert transform:
\begin{equation*}
\ha^* f=\max_{A}\bigg|\sum_{\twoline{M\text{-dyadic}}{M\ge A}}\mu_M*f\bigg|.
\end{equation*}
is of the restricted weak type (1,1).
\end{corollary}

One of the motives to investigate this question is the natural difference in the analysis of singular
integral operators in the
discrete and the continuous settings (see eg. \cite{B}, \cite{BM}, \cite{C2}, \cite{IW}, \cite{LV}, \cite{MSW}).
The proof of Theorem \ref{theorem}
is a refined version of the argument from \cite{UZ} and also employs ideas from \cite{C1}, \cite{S}.
We refer to our previous work \cite{UZ}, \cite{PZ1}, \cite{PZ2} for a more complete list of references and  motivation.

\section{Definitions.}

Let $A,\,N$ be positive dyadic integers, $\lambda>0$.
For a set $\mathcal A$ we will denote its indicator function by $1_{\mathcal A}(x)$, and its cardinality by $|\mathcal A|$.
For a function $f$ we denote
$f^{\lambda N}(x)=f(x)\cdot1_{\{|f|<\lambda N\}}(x)$. We put $f^{\lambda N}_\infty(x)=f(x)-f^{\lambda N}(x)$ and
$f^{\frac{\lambda N}{A}}=f1_{\{|f|\sim \frac{\lambda N}{A}\}}$. This last notation may be misleading, but we will try to avoid any ambiguity in what follows.

Let $\{Q\}$ be any collection of disjoint intervals in $\mathbb Z$.  We define the conditional expectation operator

\begin{equation}\label{exp_oper}
E_{\{Q\}}f(x)=\sum_Q \frac{1_Q(x)}{|Q|}\sum_{y\in Q} f(y)
\end{equation}

In particular, we will consider the family of dyadic intervals $J$ of equal length $|J|\approx M^{\theta-\epsilon}$.
Its  expectation operator \eqref{exp_oper} will be denoted by  $E_{\{J\}}$.
We will write $E_{\mathcal A}f$ if the collection contains only one interval $\mathcal A$.

We will use the following variant of the Calder\'on-Zygmund  decomposition
\begin{lemma}\label{CZ_decomp}
Let $f\ge 0$, $f\in \ell^1$, and $\lambda>0$. Then there exists a family of disjoint dyadic intervals  $Q\subset \mathbb Z$ such that
\begin{equation}\label{CZ_decom}
\lambda \le \frac{1}{|Q|}\sum_{x\in Q} f(x)\le 2\lambda,
\end{equation}
and for any $x\notin (\bigcup_Q Q)$ we have $f(x)\le \lambda$.
\end{lemma}
In what follows we will call the above intervals $Q$ Calder\'on-Zygmund cubes.
We note that for the family of Calder\'on-Zygmund cubes we have  $\|E_{\{Q\}}f(x)\|_{\ell^\infty}\le 2\lambda$.
If the family $\{Q\}$ is fixed we will abbreviate the notation $E_{\{Q\}}f(x)$ to $E_{}f(x)$.

We will estimate $|\{x:|H_M^*f_0(x)|>\lambda\}|$, where we assume $f_0\ge 0$ and $\|f_0\|_{\ell^1}=1$.
$C$ will  denote a constant, which can vary on each
of its occurrences. For $\phi\in C_c^\infty$ and a number $s$ we denote $\phi_s(x)=\phi(\frac xs)$.
If $I$ is an interval with center $x_I$, we denote $\phi_I(x)=\phi_{|I|}(x-x_I)$.

\section{Lemmas.}

We begin with the following
\begin{lemma}\label{k_reg}
Fix dyadic integers $N_1, N_2$, an interval $J$ and $\phi\in C_c^\infty$. The interval $J$ is of the size specified above, but not necessarily dyadic. Let
\begin{equation}\label{ker_def}
K_{N_1,N_2}(x_1,x_2)=\sum_{y\in J}\phi_J(y)\mu_{N_1}(x_1-y)\mu_{N_2}(x_2-y)
\end{equation}
Then
\begin{enumerate}[1.]
\setcounter{enumi}{-1}
\item $K_{N_1,N_2}(x_1,x_2)$ is supported on the set
\begin{equation*}
  \{(x_1,x_2): |x_1-x_J|\le CN_1 ^\alpha\wedge  |x_2-x_J|\le CN_2^\alpha\}.
\end{equation*}
\item If $N_2\ne  N_1$ we have
\begin{equation}\label{ker_inf}
|K_{N_1,N_2}(x_1,x_2)|\le \frac{C|J|}{(N_1N_2)^{\alpha}}
\end{equation}
and for some $\delta>0$
\begin{equation}\label{ker_reg}
|K_{N_1,N_2}(x_1+h,x_2)-K_{N_1,N_2}(x_1,x_2)|\le \frac{C|J|}{(N_1N_2)^{\alpha}}\Big(\frac{|h|}{N_1^{\alpha}}\Big)^\delta
\end{equation}
\begin{equation}\label{ker_reg_r}
|K_{N_1,N_2}(x_1,x_2+h)-K_{N_1,N_2}(x_1,x_2)|\le \frac{C|J|}{(N_1N_2)^{\alpha}}\Big(\frac{|h|}{N_2^{\alpha}}\Big)^\delta
\end{equation}

\item If $N_2= N_1=N$ then there exists a kernel $Err_{N,J}(x_1,x_2)$ supported on the set
\begin{equation*}
  \{(x_1,x_2):|x_1-x_2|\le CN^{1-\epsilon}\}
\end{equation*}
satisfying
\begin{equation}\label{3::4}
  \sum_J|Err_{N,J}(x_1,x_2)|\le \frac{C}{N^{\alpha}}
\end{equation}
such that
\begin{equation}\label{ker_reg_1}
K_{N,N}(x_1,x_2)=\tilde K_{N,N}(x_1,x_2) +Err_{N,J}(x_1,x_2)+C_{N,J}\delta_0(x_1-x_2)
\end{equation}
where
\begin{equation*}
  |C_{N,J}|\le \frac{C|J|}{N^{1+\alpha}}
\end{equation*}
and the kernel $\tilde K_{N,N}(x_1,x_2)$ satisfies \eqref{ker_inf}, \eqref{ker_reg}, \eqref{ker_reg_r}.

\item Let $Q$ be an interval of length $|Q|\ge M^{\alpha-1+\epsilon}$, for some $\epsilon>0$, and let
\begin{equation*}
  K_N(x)=\frac1{|Q|}\mu_N*1_Q(x).
\end{equation*}
Then, for some $\delta>0$
\begin{equation}\label{kc_reg}
\sum_x |K_N(x+h)-K_N(x)|^2\le  \frac{C}{N^{\alpha}}\Big(\frac{|h|}{N^{\alpha}}\Big)^\delta
\end{equation}
\end{enumerate}
\end{lemma}

\begin{proof}
Statement 0 is obvious. Statements 1 and 2 follow by inspection
of the proofs of the Lemma (3.1) in \cite{UZ} and the  Lemma in the appendix of \cite{PZ2}.
We briefly sketch the necessary modifications of the argument.
Since the statement is translation invariant we can assume that $J$ is centered at 0, with length $\sim M^{\theta-\epsilon}$. We can assume $N_1\ge N_2$.
Let
\begin{equation*}
  Q=N_1^{1-\delta_1},\quad R=Q^{\delta_2},\quad H=Q^{\delta_3}
\end{equation*}
where ${\delta_1}, {\delta_2}, {\delta_3}$ are small positive constants, which will be specified later. We define
\begin{equation*}
  Z_{jQ}=\frac{1}{\alpha\big|x_1-x_2-(jQ)^\alpha\big|^\frac{\alpha-1}{\alpha}}.
\end{equation*}
One can easily check that under the conditions of the lemma the denominator never vanishes.
In fact one can check, that if
\begin{equation}\label{3::5}
   \alpha-\delta_1<\alpha \theta
\end{equation}
then the denominator of $Z_{jQ}$ is comparable to $N_2^\alpha$.
It what follows we will fix $\delta_1=\frac{1}{100}$. That, together with the choice of $\theta$
close to 1 will ensure \eqref{3::5}. $\theta$ can be assumed to be arbitrarily close to 1 without
loss of generality (see note in the Appendix of \cite{PZ2}). We define 1 - periodic functions
$\tilde\Psi_a$ and $\Psi_a$ by the conditions:
\begin{equation}
0\le \Psi_{jQ}\le1, \quad \Psi_{jQ}(t)=1 \text{ for } 1-Z_{jQ}\le t \le 1
\end{equation}

\begin{equation}
\text{supp}( \Psi_{jQ}(t))\subset  \Big\{t: 1-Z_{jQ}\Big(1+\frac1{R}\Big)\le t \le 1+Z_{jQ}\frac{1}{R}\Big\} +\zet,
\end{equation}
where $\zet$ is the set of integers,
\begin{equation}
  \text{ } \Psi_{jQ}(t)\in C^\infty  \text{ and  }
 |\partial^k \Psi_{jQ}(t)|\le C(RZ_{jQ}^{-1})^k \text{ for   }k\le 4
\end{equation}

\begin{equation}
0\le \widetilde\Psi_{H}\le 1,\text{ }
\widetilde\Psi_{H}\in C^\infty, \text{  } |\partial^k\widetilde\Psi_{H}|\le CH^k \text{ for } k\le 4
\end{equation}

\begin{equation}
\text{ supp}\widetilde\Psi_{H}\subset \Big[0,\frac2H\Big]+\zet \text{ and }\sum_{0\le P <H}\widetilde\Psi_{H}\Big(\Theta-\frac{P}{H}\Big)=1
\end{equation}

Now, we briefly adopt the following convention: The symbol $a\le_N b$ means $b-a=O(\frac1{N^{\alpha+\eta}})+C$
where $\eta>0$ is some constant depending on $\alpha$ and  independent of $N$, and $C>0$.
Similarly as in \cite{UZ} we have
\begin{align}\label{ker_inf_10}
&K_{N_1,N_2}(x_1,x_2)\le_{N_1}\\
&\le_{N_1}\sum_{0\le P <H}\sum_{j}\sum_{s=0}^{Q-1}\widetilde\Psi_{H}\Big((jQ+s)^\alpha-\frac PH\Big)\times\notag\\
&\qquad\times\Psi_{jQ}\Big(\Big(x_1-x_2-(jQ+s)^\alpha-\frac PH\Big)^{\frac1\alpha}\Big)\times\notag\\
&\qquad\times\phi_{N_1^\alpha}((jQ+s)^\alpha)\phi_{N_2^\alpha}(x_1-(jQ+s)^\alpha-x_2)\phi_J(x_1-(jQ+s)^\alpha)\notag
\end{align}

Now we replace $\Psi_{jQ}$ by its Fourier series.
By the definition, we have
\begin{multline*}
{c_\alpha}\Big(1-\frac1R\Big){(x_1-x_2-(jQ)^\alpha)^{-\frac{\alpha-1}{\alpha}}}\le\\
\le\widehat{\Psi}_{jQ}(0)\le {c_\alpha}\Big(1+\frac1R\Big){(x_1-x_2-(jQ)^\alpha)^{-\frac{\alpha-1}{\alpha}}}.
\end{multline*}

Since $\sum_{0\le P <H}\widetilde\Psi_{H}(\Theta-\frac{P}{H})=1$, the
right hand side expression corresponding to  $\widehat\Psi_{jQ}(0)$ becomes
\begin{align*}
&D(x_1,x_2)=\\
&\quad=\sum_{j,s}\widehat\Psi_{jQ}(0)
\phi_{N_1^\alpha}((jQ+s)^\alpha)\phi_{N_2^\alpha}(x_1-(jQ+s)^\alpha-x_2)\times\\
&\qquad\qquad\qquad\times\phi_J(x_1-(jQ+s)^\alpha)\\
&\quad\le\Big(1+\frac 1R\Big)\sum_{m}\frac{c_\alpha}{(x_1-x_2-(jQ)^\alpha)^{\frac{\alpha-1}{\alpha}}}
\phi_{N_1^\alpha}(m^\alpha)\phi_{N_2^\alpha}(x_1-m^\alpha-x_2)\times\\
&\qquad\qquad\qquad\times\phi_J(x_1-m^\alpha)\\
&\quad\le\Big(1+\frac 1R\Big)\int_{R_+}\frac{c_\alpha}{(x_1-x_2- t^{\alpha})^{\frac{\alpha-1}{\alpha}}}
\phi_{N_1^\alpha}(t^\alpha)\phi_{N_2^\alpha}(x_1-t^\alpha-x_2)\times\\
&\qquad\qquad\qquad\times\phi_J(x_1-t^\alpha)dt\\
&\quad=\Big(1+\frac 1R\Big)|J|\int_{R_+}\frac{c_\alpha}{ (x_2-|J|u)^{\frac{\alpha-1}{\alpha}}}
\phi_{N_1^\alpha}(x_1-|J|u)\phi_{N_2^\alpha}(|J|u-x_2)\times\\
&\qquad\qquad\qquad\times\phi(u)\frac{\alpha du}{(x_1-|J|u)^{\frac{\alpha-1}{\alpha}}}\\
&\quad=\Big(1+\frac 1R\Big)|J|F(x_1,x_2)
\end{align*}
By similar arguments, one can obtain the lower estimate
\begin{equation*}
D(x_1,x_2)\ge \Big(1-\frac 1R\Big)|J|F(x_1,x_2)
\end{equation*}
Now the function of $F(x_1,x_2)$  can be easily shown to satisfy the estimates
$|\partial_{x_1}F(x_1,x_2)|\le \frac{C}{N_1^{2\alpha-1}N_2^{\alpha-1}}$. This immediately yields
\eqref{ker_inf}, \eqref{ker_reg}, \eqref{ker_reg_1}.

In order to complete the argument we need to estimate the
summands  corresponding to coefficients $\widehat\Psi_{jQ}(k)$, $k\ne 0$.

For $j,P$ fixed, we are  left with the estimates for
\begin{align*}
&\sum_{s=0}^{Q-1}
\widetilde\Psi_{H}\Big((jQ+s)^\alpha-\frac PH\Big)\times\\
&\qquad\times\Big(\Psi_{jQ}\Big(\Big(x_1-x_2-(jQ+s)^\alpha-\frac PH\Big)^{\frac1\alpha}\Big)-\widehat\Psi_{jQ}(0)\Big)
\times\\
&\qquad\times\phi_{N_1^\alpha}((jQ+s)^\alpha)\phi_{N_2^\alpha}(x_1-(jQ+s)^\alpha-x_2)\phi_J(x_1-(jQ+s)^\alpha)
\end{align*}
Let $N_1\ne N_2$ or $|x_1-x_2|\ge CN_1$, and  $|J|\ge Q$. Then the above sum is exactly of the form
considered in \cite{UZ}. Let $\delta_1=\frac{1}{100}$. Using van der Corput Lemma we get an
estimate with additional $Q^{-\eta}$ factor, where $\eta>0$ is independent of $\alpha$. We fix
$\delta_2,\delta_3$ sufficiently small independent of $\alpha$. Then summing with respect to $j, P$
completes the job.  We note, that choosing $\theta$ sufficiently close to $1$ we can ensure
required $|J|\ge Q$. We omit further details, and refer the reader to \cite{UZ}.

The estimate \eqref{3::4}
follows from
\begin{equation*}
  |\mu_{N}*\check\mu_{N}(x)|\le  \frac{C}{N^\alpha} \text{ for }x\ne 0,
\end{equation*}
proved in \cite{UZ}.

The H\"older's estimate \eqref{kc_reg} has been  proved in \cite{PZ1}, Lemma 3.5. The proof there is carried out for a smooth function $\varphi$ in the place of characteristic function, but it carries over (with small modification in part III, see \cite{PZ1}).
\end{proof}

We fix small $\epsilon>0$, $1-\epsilon<\theta<1$, the function $f_0\ge 0$,  $\lambda >0$ and  the set of the Calder\'on-Zygmund
cubes $\{Q\}$ associated with $f_0$ by Lemma \ref{CZ_decomp}.
 By the $\ell^2$ boundedness of the maximal rough Hilbert transform $H_M^*$ and Lemma \ref{CZ_decomp}
we can assume that $f_0=0$
away from the $\bigcup_Q Q$.  Now we modify $f_0$ putting
\begin{equation}\label{Q_size_as}
f_0=0 \text{ on each Calder\'on-Zygmund
cube with }|Q|\le M^{\theta-2\epsilon}.
\end{equation}

We will denote new function again by $f_0$.  In the remark at the end of the paper we explain
why this procedure do not bring any loss of the generality.

We perform further reductions.
First we have
\begin{align*}
&\max_{B}|\sum_{N\le B}(\mu_N-\check\mu_N)*f_0|\le
\max_{B}|\sum_{N\le B}(\mu_N-\check\mu_N)*(f_{0,\infty}^{\lambda N})|+ \\
&\qquad\qquad+\max_{B}|\sum_{N\le B}(\mu_N-\check\mu_N)*(Ef_{0,\infty}^{\lambda N})|+\\
&\qquad\qquad+\max_{B}|\sum_{N\le B}(\mu_N-\check\mu_N)*(f_0^{\lambda N}-Ef_0^{\lambda N})|+\\
&\qquad\qquad+\max_{B}|\sum_{N\le B}(\mu_N-\check\mu_N)*Ef_0|\\
&\qquad=I+II+III+IV
\end{align*}
Now, $I$ has been estimated in \cite{PZ2}, \cite{UZ} using support properties.  The estimates for the $IV$ follow since $|Ef_0|\le C\lambda$
and the maximal Hilbert transform $H_M^*$ is bounded on  $\ell^2$.
The term $II$ requires some care. First we note, that
$G_M=\sum_{M^\theta \le N \le M}(\mu_N-\check\mu_N)*(Ef_{0,\infty}^{\lambda N})\in \ell^2$
and $\|G_M\|^2_{\ell^2}\le C\lambda \|f_0\|_{\ell^1}$.
This has been proved in \cite{PZ2} page 22 with $f_{0,\infty}^{\lambda N}$ replaced by $f_{0}^{\lambda N}$,
The proof of our statement is exactly the same, and we do not present any details.

Next,
the key observation is the following property of  the Calder\'on-Zygmund cubes  in our setting
\begin{lemma}\label{key_CZ}
\begin{enumerate}[1.]
\item Let $Q$ be  the Calder\'on-Zygmund cube which contains the point of the support of $f_{0,\infty}^{\lambda N}$.
Then  $|Q|\ge CN$.
\item Let $Q$ be  the Calder\'on-Zygmund cube which contains the point of the support of $f_0^{\frac{\lambda N}{A}}$.
Then  $|Q|\ge CNA^{-1}$.
\end{enumerate}
\end{lemma}
\begin{proof}
The lemma follows immediately from the upper inequality in \eqref{CZ_decom}.
\end{proof}

From \eqref{Q_size_as} and  \eqref{kc_reg} we infer that
  the function $\mu_N*1_Q$ is $\ell^2$ is H\"older regular.  Consequently we can estimate $II$ in the similar
manner as the maximal Calder\'on-Zygmund operator. First observe that by \eqref{kc_reg} the following estimates hold
\begin{equation*}
\sum_{N\le B}|\phi_{B^\alpha}*(\mu_N-\check \mu_N)(x)|\le C\phi_{4B^\alpha}(x)
\end{equation*}
\begin{equation*}
\Big|\big(\mu_N-\phi_{B^\alpha}*\mu_N\big)*\frac{1_Q}{|Q|}(x)\Big|\le C\Big(\frac{B}{N}\Big)^\delta\phi_{4N^\alpha}(x)\text{ for  } B\le N
\end{equation*}
and consequently
\begin{multline*}
\Big|\sum_{ N \ge B}(\mu_N-\check\mu_N)*(Ef_{0,\infty}^{\lambda N}(x)-\phi_{B^\alpha}*G_M(x))\Big|\le \\
\le C M(Ef_0)(x) +C( M((Ef_0)^2))^\frac12(x)
\end{multline*}
where $M$ denotes the classical Hardy-Littlewood maximal function. By weak type (1,1) of $M$ we obtain $II\le \frac C\lambda$.
We leave the details (which are standard arguments in the Calder\'on-Zygmund theory) to the reader.

So the main term is $III$. We further decompose $f_0$
writing down
\begin{equation*}
  f_0^{\lambda N}=\sum_{A\text{-dyadic}} f_0^{\frac{\lambda N}{A}}
\end{equation*}
and obtain the following estimate
\begin{equation}\label{}
|III|\le\max_{B}|\sum_{N\le B}\sum_{1\le A\le N}(\mu_N-\check\mu_N)*(f_0^{\frac{\lambda N}{A}}-Ef_0^{\frac{\lambda N}{A}})|
\end{equation}
For $i=1,2$ we write
\begin{equation}\label{bsi_def}
b^{A,N}_{s,i}(x)=\sum_{Q\in \mathcal D_{A,N,s}^i}
(1_{Q}(x)f_0^{\frac{\lambda N}{A}}(x)-E_{Q}f_0^{\frac{\lambda N}{A}}(x)),
\end{equation}
\begin{equation}\label{fsi_def}
f^{A,N}_{s,i}(x)=\sum_{Q\in \mathcal D_{A,N,s}^i}
1_{Q}(x)f_0^{\frac{\lambda N}{A}}(x),
\end{equation}
where the families of cubes $ \mathcal D_{A,N,s}^i$, $i=1,2$ are defined as follows
\begin{align}\label{D1_def}
 \mathcal D_{A,N,s}^1&=\{Q: |Q|\sim(2^{-s}N)^\alpha,\,2^s> A\}\notag\\
 \mathcal D_{A,N,s}^2&=\{Q: |Q|\sim(2^{-s}N)^\alpha,\,2^s\le A\}.
\end{align}
Furthermore, for $i=1,2$ we let
\begin{equation}\label{fi_def}
f^{A,N}_{i}(x)=\sum_{s\ge0} f^{A,N}_{s,i}(x),
\end{equation}
and for $i=2$
\begin{equation}\label{3::3}
  f^N_{s,2}(x)=\sum_{A:\,A\ge2^{s}}f^{A,N}_{s,2}(x).
\end{equation}

From the above definitions, for fixed $A,N$ we have
\begin{equation}\label{}
\bigcup_{s\ge 0}( \mathcal D_{A,N,s}^1\cup \mathcal D_{A,N,s}^2)=\{Q: |Q|\le N^\alpha\}
\end{equation}

We note that $ \text{supp} \mu_N*1_{Q}\subset Q^{**}$
if $|Q|\ge N^\alpha$, and by \eqref{CZ_decom} we have $\sum |Q^{**}|\le \frac1\lambda$.
Hence it will suffice to estimate

\begin{equation}\label{3::1}
H_i^*(x)=\max_{B}|\sum_{N\le B}\sum_{A}\sum_{s\ge 0}(\mu_N-\check\mu_N)*b^{A,N}_{s,i}(x)|
\end{equation}

\begin{lemma}\label{max_sparse_est}
Let $f^{A,N}_{i}$ be defined in \eqref{fi_def} and let $\beta_N\ge 0$ be a sequence of numbers. We define a sequence of integers $N_j$,
$1\le j\le j_{max}^A$ by
\begin{equation}\label{n_jump_def}
N_j=\max\big\{2^k:\sum_{N\le 2^k}\beta_N\le j\,\lambda_0\big\}
\end{equation}
whenever the maximum exists. Assuming
\begin{equation*}
  \beta_N\le \lambda_0
\end{equation*}
we see that $N_j$'s form a strictly increasing sequence of dyadic integers.
Let  $x\in J$ be such that
\begin{equation}\label{max_cond_1}
\max_{B}|\sum_{N\le B } \mu_N*f^{A,N}_{i}-\sum_{N\le B} \beta_N|\ge 4\lambda_0
\end{equation}
Then
\begin{equation}
\max_{j}|\sum_{N\le N_j } \mu_N*f^{A,N}_{i}-\sum_{N\le N_j} \beta_N|\ge \lambda_0
\end{equation}
The same is true for the functions $f^N_{s,2}$ in the place of $f^{A,N}_{i}$. We call $j_{max}^s$ the largest value of $j$ in this case.
\end{lemma}
\begin{proof}
Fix $B$ maximising the estimate in \eqref{max_cond_1} and the unique $j$ such that $N_j<B\le N_{j+1}$. Assume, that
\begin{equation}\label{max_cond_1_1}
|\sum_{N\le N_j } \mu_N*f^{A,N}_{i}-\sum_{N\le N_j} \beta_N|\le \lambda_0
\end{equation}
Then
\begin{equation}
|\sum_{N_j<N\le B } \mu_N*f^{A,N}_{i}-\sum_{N_j<N\le B} \beta_N|\ge 3\lambda_0
\end{equation}
and by \eqref{n_jump_def} we must have
\begin{equation}
\sum_{N_j<N\le B } \mu_N*f^{A,N}_{i}-\sum_{N_j<N\le B} \beta_N\ge 3\lambda_0
\end{equation}
Since $f^{A,N}_{i}\ge 0$, again by \eqref{n_jump_def}
\begin{multline*}
\sum_{N_j<N\le N_{j+1} } \mu_N*f^{A,N}_{i}-\sum_{N_j<N\le N_{j+1}} \beta_N\ge\\
\ge\sum_{N_j<N\le B } \mu_N*f^{A,N}_{i}-\sum_{N_j<N\le B} \beta_N-{\lambda_0}\ge 2\lambda_0
\end{multline*}
Applying \eqref{max_cond_1_1} we get
\begin{equation}
|\sum_{N\le  N_{j+1}  } \mu_N*f^{A,N}_{i}-\sum_{N\le  N_{j+1} } \beta_N|\ge \lambda_0
\end{equation}
The proof for functions $f^N_{s,2}$ is similar.
\end{proof}

\noindent{\bf Remark.} Under the assumptions of the above lemma we can split the set of dyadic naturals into two collections
\begin{equation*}
  \mathcal C_1=\{N: \beta_N\le \lambda_0\}\qquad\text{and}\qquad
  \mathcal C_2=\{N: \beta_N>\lambda_0\},
\end{equation*}
such that
\begin{equation}\label{max_cond_1-cont}
\max_{j}\Big|\sum_{N\le N_j, N\in \mathcal C_1 } \mu_N*f^{A,N}_{i}-\sum_{N\le N_j, N\in \mathcal C_1} \beta_N\Big|\ge \frac12 \lambda_0
\end{equation}
or
\begin{equation}\label{}
\max_{j}|\sum_{N\le N_j, N\in \mathcal C_2 } \mu_N*f^{A,N}_{i}-\sum_{N\le N_j, N\in \mathcal C_2} \beta_N|\ge \frac12 \lambda_0
\end{equation}
and the cardinality of $|\mathcal C_2|\le j_{max}^A$.

\begin{lemma}\label{max_cs}
Fix an interval $J$, $A$
and $i$. Let $b^{A,N}_{s,i}$ be defined in \eqref{bsi_def}, $x\in J$ and
\begin{equation}
\max_{B}|\sum_{N\le B }\sum_{s} \mu_N*b^{A,N}_{s,i}(x)|\ge 4\lambda_0
\end{equation}
Then there exists the sequence of nonegative numbers $\beta_N$ independent on $x\in J$, such that \eqref{max_cond_1} holds, or
(the definition of $ER$  below) $ER(x)\ge \lambda_0$.
The sequence $\beta_N$  depends on $J$.

If $\{N_j\}$ is a sequence of numbers such that \eqref{max_cond_1} holds for \eqref{beta_def}
then
\begin{equation}
\max_{j}|\sum_{N\le N_j }\sum_{s} \mu_N*b^{A,N}_{s,i}(x)|\ge \frac{\lambda_0}8
\end{equation}
or $ER(x)\ge \lambda_0$.
\end{lemma}

\begin{proof}
By the definition we have $\sum_s \mu_N*b^{A,N}_{s,i}(x)= \mu_N*f^{A,N}_{i}(x)- F_{A,N}(x)$.
where $ F_{A,N}(x)= \mu_N*Ef^{A,N}_{i}(x)$.
Then we put
\begin{equation}\label{beta_def}
\beta_N=E_{\{J\}} F_{A,N}(x)
\end{equation}
 where  $E_{\{J\}}$ is the  conditional expectation operator
defined by \eqref{exp_oper}, and $x\in J$.
By \eqref{Q_size_as} we have $|Q|\ge N^{1-3\epsilon}$.
We will show that the error function
\begin{equation}
ER(x)=\sum_{N}| F_{A,N}(x)-E_{\{J\}} F_{A,N}(x)|
\end{equation}
satisfies
\begin{equation}\label{3::2}
\|ER\|_{\ell^1}\le CM^{-\delta_e}\|f_0\|_{\ell^1}.
\end{equation}
where $\delta_e$ is some small constant depending on the $\epsilon$ in the definition of $J$. Observe
\begin{align*}
  F_{A,N}(x)&=
  \mu_N*\Big(\sum_{Q}\frac{1_Q}{|Q|}\sum_Qf_i^{A,N}\Big)\ (x)\\
  &=\sum_Q\frac{C_{N,Q}}{|Q|}\ \mu_N*1_Q(x)\\
  &=\sum_Q\rho_{N,Q}(x)\cdot C_{N,Q},
\end{align*}
where
\begin{equation*}
  \rho_{N,Q}=\mu_N*\frac{1_Q}{|Q|},\qquad C_{N,Q}=\sum_{x\in Q}f_i^{A,N}(x).
\end{equation*}
Consequently, for $x\in J$
\begin{align*}
  ER(x)&\le\sum_N\sum_QC_{N,Q}\Big|E_{\{J\}}\rho_{N,Q}(x)- \rho_{N,Q}(x)\Big|\\
  &\le\sum_N\sum_Q\,C_{N,Q}\,\frac{1_J(x)}{|J|}\sum_{h\in J}\big|\rho_{N,Q}(h)-\rho_{N,Q}(x)\big|\\
  &\le\sum_N\sum_Q\,C_{N,Q}\,\frac{1_J(x)}{|J|}\sum_{h\in J_0^*}\big|\rho_{N,Q}(h+x)-\rho_{N,Q}(x)\big|,
\end{align*}
where $J_0^*$ is the cube centered at 0, with size double that of $J$. Thus (recall, that $|J|\sim M^{\theta-\epsilon}$)
\begin{align*}
 \|ER\|_{\ell^1}&\le\sum_N\sum_Q\,\frac{C_{N,Q}}{|J|}\sum_J\sum_{x\in J}\sum_{h\in J_0^*}\big|\rho_{N,Q}(h+x)-\rho_{N,Q}(x)\big|\\
 &\le \frac{1}{|J|}\sum_N\sum_Q C_{N,Q}\sum_{h\in J_0^*}\sum_{x}\big|\rho_{N,Q}(h+x)-\rho_{N,Q}(x)\big|\\
 &\le C \frac{1}{|J|}\sum_N\sum_Q C_{N,Q}\sum_{h\in J_0^*}\Big(\frac{|h|}{N^\alpha}\Big)^\delta\\
 &\le C \Big(\frac{M^{\theta-\epsilon}}{M^{\theta\alpha}}\Big)^\delta\sum_N\sum_Q C_{N,Q}\\
 &\le C M^{-\epsilon\delta}\,\|f_0\|_{\ell^1}.
\end{align*}
We have used an $\ell^1$ H\"older estimate for $\rho_{N,Q}$. It follows from the $\ell^2$ estimate \eqref{kc_reg} by Cauchy-Schwarz inequality. We have also used
\begin{equation*}
  \sum_{N,Q} C_{N,Q}\le \|f_0\|_{\ell^1}.
\end{equation*}

Finally, for the case $i=2$ we define, similarly to $F_{A,N}$,
\begin{equation*}
  F^N_{s,2}(x)=\mu_N*Ef^N_{s,2}(x),
\end{equation*}
where $f^N_{s,2}$ was defined in \eqref{3::3}.
The resulting $\beta_N=E_{\{J\}} F^N_{s,2}(x)$ and the error function $ER$ can be treated in the same way.
\end{proof}
\begin{lemma}\label{jmax_est}
Let
 $$\mathcal A_A=\{x\in Z: \sum_{N}E_{\{J\}} F_{A,N}(x)\ge \lambda A^2\},$$
 $$\mathcal A^s=\{x\in Z: \sum_{N}E_{\{J\}} F^{N}_{s,2}(x)\ge \lambda 2^{s\epsilon}\}.$$
Then
we have $|\mathcal A_A|\le \frac{1}{\lambda A^2}$,
 $|\mathcal A^s|\le \frac{1}{\lambda 2^{s\epsilon}}$,
moreover for each $J$, $J\subset \mathcal A_A$ or   $J\cap \mathcal A_A=\emptyset $ and
$J\subset \mathcal A^s$ or   $J\cap \mathcal A^s=\emptyset $.

\end{lemma}
\begin{proof}
Immediate from
the Chebyshev inquality. The second part follows since $E_{\{J\}} F_{A,N}$ is constant on each $J$.
\end{proof}
By the above lemma \ref{jmax_est} we can consider only the intervals $J$ with $j_{max}^A\le CA^3$
and   $j_{max}^s\le 2^{3s\epsilon}$ for every $A,s$. This will allow
us to apply Menschov classical estimate of the maximal function by the sum of small number of the square functions.

\begin{lemma}\label{menshov}
Let $\{a_i\}_{1\le i \le D}$ be a sequence of numbers. Then
\begin{equation}\label{menshov_est}
\max_{1\le i\le D}|\sum_{1\le j \le i}a_j|^2\le \log D\sum_{k\le \log D} \sum_{s\le  D}|\sum_{2^k s\le j \le 2^k (s+1)}a_j|^2
\end{equation}
\end{lemma}
\begin{proof} This is a well known fact, see eg. \cite{KS}.
\end{proof}
For fixed $N$ we denote by $\mathcal J_N$ the family of dyadic intervals $I$ of size $8N^\alpha\le |I| < 16N^\alpha$.
Then, for $N$ fixed,  $\cup_{I\in\mathcal J_N}I=\mathbb Z$. Moreover, any two $I_1\in\mathcal J_{N_1}, I_2\in\mathcal J_{N_2}$ either have empty intersection, or one is a subset of the other. For given interval $I$ we denote by $I^*$ an interval concentric with $I$, with a larger diameter. The exact ratio of diameters depends on constants appearing in Lemma \ref{k_reg}, and will be obvious from the context.

\begin{lemma}\label{sqr_UZ}
 Let $b^{A,N}_{s,i}$ be defined as in \eqref{bsi_def}.  Fix $J$ and $i$.
 Let $I_1\in\mathcal J_{N_1},I_2\in\mathcal J_{N_2}$,  $J\subset I_{N_1}\cap I_{N_2}$.
 Then for any fixed increasing sequence of integers $\{S_j\}$ we have
\begin{align}\label{square_1}
&\sum_{y\in J}\sum_{j}|\sum_s\sum_{S_j< N\le S_{j+1}} \mu_N*b^{A,N}_{s,i}(y)|^2\phi_J(y)\le\notag\\
&\qquad\le\sum_{s_1,s_2} 2^{-\delta (s_1+s_2)} \sum_{N_2\le N_1}\frac{|J|}{( N_1N_2)^{\alpha}}
\|b^{A,N_1}_{s_1,i}\|_{\ell^1(I_1^*)}\|b^{A,N_2}_{s_2,i}\|_{\ell^1(I_2^*)}+\\
&\qquad\qquad+\sum_{N_1,s}\frac{|J|}{ N_1^{\alpha+1}}\sum_{x\in I_1^*}|b^{A,N_1}_{s,i}(x)|^2+\notag\\
&\qquad\qquad+\sum_{N_1,s_1,s_2}\langle\,|Err_{N_1,J}\,b^{A,N_1}_{s_1,i}|,| b^{A,N_1}_{s_2,i}|\,\rangle\notag\\
&\qquad=D_I(J)+D_{II}(J)+D_{III}(J)\notag
\end{align}
The terms $D_{II}(J),D_{III}(J)$ appears only if $N_1=N_2$. Moreover the RHS of \eqref{square_1} do not depend on the particular
choice of the sequence $S_j$.
\end{lemma}
\begin{proof}
We expand the square as a double sum
\begin{equation}\label{square_1_01}
 \begin{aligned}
\sum_{y\in J}\sum_{j}|\sum_s\sum_{S_j< N\le S_{j+1}} \mu_N*b^{A,N}_{s,i}(y)|^2\phi_J(y)=\\
\sum_{s_1,s_2}\sum_j  \sum_{S_j<N_1, N_2\le S_{j+1}}
<K_{N_1,N_2}^J b^{A,N_1}_{s_1,i}, b^{A,N_2}_{s_2,i}>
\end{aligned}
\end{equation}
and  to each summand we apply the regularity estimate \eqref{ker_inf}, \eqref{ker_reg} of $K_{N_1,N_2}^J$, see Lemma \ref{k_reg}.

Let $s_1\ge s_2$. If $N_1\ne N_2$ then we apply \eqref{ker_reg} and the  standard cancellation argument (we omit the details).
This  leads to
\begin{equation}
\|K_{N_1,N_2}^J b^{A,N_1}_{s_1,i}\|_{\ell^\infty}\le
2^{-s_1\delta}|J|(N_1N_2)^{-\alpha}\|b^{A,N_1}_{s_1,i}\|_{\ell^1(I_1^*)}
\end{equation}

If $N_1= N_2$, we obtain the same estimate for $\tilde K$ instead of $K$.

If $s_1< s_2$ we repeat the  above argument to the kernels  conjugate to $K, \tilde K$, acting on
$b^{A,N_2}_{s_2,i}$.  The estimate of the first summand in \eqref{square_1} follows.

If $N_1=N_2$, $s_1\ne s_2$
then the supports of the functions $ b^{A,N_1}_{s_1,i}$,  $ b^{A,N_2}_{s_2,i}$ are disjoint.
Consequently, the $\delta_0$ term in \eqref{ker_reg_1}  produce the second $D_{II}(J)$ term in \eqref{square_1},
and the Lemma follows.
\end{proof}

\section{Proof of Theorem \ref{theorem}.}
We  now return to the proof of Theorem \ref{theorem}. Recall, that we have reduced the proof to estimating the 2 operators $H_i^*$ given by \eqref{3::1}, $i=1,2$. We proceed with the 2 cases.
\newline {\bf Case $i=1$.}  Recall the definition \eqref{bsi_def}:
\begin{equation*}
b^{A,N}_{s,1}(x)=\sum_{Q\in \mathcal D_{A,N,s}^1}
(1_{Q}(x)f_0^{\frac{\lambda N}{A}}(x)-E_{Q}f_0^{\frac{\lambda N}{A}}(x))
\end{equation*}
For a dyadic interval $J$ (recall that we only consider $J$'s such that $|J|\sim M^{\theta-\epsilon}$) we need to estimate
\begin{align*}
&\Big|\Big\{x\in J:\max_B\Big|\sum_{N\le B}\sum_{s\ge0}(\mu_N-\check\mu_N)*b^{A,N}_{s,1}(x)\Big|>\lambda A^{-\epsilon}\Big\}\Big| \\
&\qquad\qquad\le\Big|\Big\{x\in J:\max_B\Big|\sum_{N\le B}\sum_{s\ge0}\mu_N*b^{A,N}_{s,1}(x)\Big|>\frac12\lambda A^{-\epsilon}\Big\}\Big| +\\
&\qquad\qquad\quad+\Big|\Big\{x\in J:\max_B\Big|\sum_{N\le B}\sum_{s\ge0}\check\mu_N*b^{A,N}_{s,1}(x)\Big|>\frac12\lambda A^{-\epsilon}\Big\}\Big|.
\end{align*}
We call the first summand $L_J$. It is enough to estimate $L_J$, since the second summand can be estimated analogously. We apply Lemma \ref{max_cs} (with $\lambda_0=c\lambda A^{-\epsilon}$) and thus there exists a sequence $\{N_j\}_{j\le cA^3}$, depending on $J$ but independent of $x\in J$, the collections $\mathcal C_1, \mathcal C_2$ such that for some  $v\in \{1,2\}$ ($v=v(A,J)$
\begin{align*}
L_J&\le\Big|\Big\{x\in J:\max_{j\le cA^3}\Big|\sum_{N\le N_j, N\in C_v}\sum_{s\ge0}\mu_N*b^{A,N}_{s,1}(x)\Big|>\lambda A^{-\epsilon}\Big\}\Big|\\
&\qquad+|\{x\in J:ER(x)>c\lambda A^{-\epsilon}\}|+|\mathcal A_A\cap J|,
\end{align*}
(regarding the range of $j$'s see Lemma \ref{jmax_est} and remark that follows). Now apply Lemma \ref{menshov}, and obtain a $k\le c\log A$ such that if we put $S_j=N_{2^kj}$ we have

\begin{align*}
\sum_JL_J&\le\sum_J\Big|\Big\{x\in J:\sum_{j}\Big|\sum_{S_j<N\le S_{j+1}}\sum_{s\ge0}\mu_N*b^{A,N}_{s,1}(x)\Big|^2>\lambda^2 A^{-3\epsilon}\Big\}\Big|\\
&\qquad+|\{x:ER(x)>c\lambda A^{-\epsilon}\}|+|\mathcal A_A|.
\end{align*}
By \eqref{3::2} (and Chebychev's inequality), Lemmas \ref{max_cs} and \ref{jmax_est}, the two last summands have estimates
\begin{gather*}
|\{x:ER(x)>c\lambda A^{-\epsilon}\}|\le c\,\frac{\|f_0\|_{\ell^1}A^{\epsilon}}{\lambda N^\delta} \le c\,\frac{1}{\lambda A^{\epsilon'}},\\
|\mathcal A_A|\le c\,\frac{1}{\lambda A^2},
\end{gather*}
since $\epsilon$ can be chosen small enough and, as was pointed out before, only $A\le N$ are relevant.
We turn to the first summand. Applying Lemma \ref{sqr_UZ} and the Chebychev's inequality we get
\begin{align}\label{4::1}
&\sum_J\Big|\Big\{x\in J:\sum_{j}\Big|\sum_{\twoline{S_j<N\le S_{j+1}}{N\in\mathcal C_v}}\sum_{s\ge0}\mu_N*b^{A,N}_{s,1}(x)\Big|^2>\lambda^2 A^{-3\epsilon}\Big\}\Big|\notag\\
&\quad\le\sum_{\twoline{s_i:2^{s_i}\ge A^{}}{i=1,2}}\sum_J\sum_{N_1\ge N_2}2^{-\delta(s_1+s_2)}
\frac{A^{3\epsilon}|J|}{\lambda^2 (N_1N_2)^{\alpha}}\times\notag\\
&\qquad\qquad\qquad\times
\|b^{A,N_1}_{s_1,1}\|_{\ell^1(I_{N_1}^*(J))}\|b^{A,N_2}_{s_2,1}\|_{\ell^1(I_{N_2}^*(J))}+\notag\\
&\qquad+\sum_J\sum_{N}\sum_{s:2^s\ge A^{}}\frac{A^{3\epsilon}|J|}{\lambda^2 N^{\alpha+1}}\|b^{A,N}_{s,1}\|_{\ell^2(I_{N}^*(J))}^2\\
&\qquad+\sum_J\sum_{\twoline{s_i:2^{s_i}\ge A^{}}{i=1,2}}\sum_{N}\frac{A^{3\epsilon}}{\lambda^2}\big|\langle Err_{N,J}b^{A,N}_{s_1,1},b^{A,N}_{s_2,1}\rangle\big|\notag\\
&\quad=I+II+III,\notag
\end{align}
where $I_{N}(J)$ is the unique dyadic interval from the family $\mathcal J_N$ which contains $J$. Again, we start with the first component $I$. Note, that the sum of $|J|$ of those $J$'s, which share the same $I_{N_2}(J)$ is equal to $|I_{N_2}(J)|\sim N_2^\alpha$.
So,
\begin{multline*}
I\le\sum_{\twoline{s_i:2^{s_i}\ge A^{}}{i=1,2}}2^{-\delta(s_1+s_2)}\sum_{N_1\ge N_2}\sum_{I_{N_1}\subset\mathcal J_{N_1}}
\sum_{\twoline{I_{N_2}\subset\mathcal J_{N_2}}{I_{N_2}\subset I_{N_1}}}\frac{A^{3\epsilon}}{\lambda^2 N_1^{\alpha}}\times\\
\times\|b^{A,N_1}_{s_1,1}\|_{\ell^1(I_{N_1}^*)}\|b^{A,N_2}_{s_2,1}\|_{\ell^1(I_{N_2}^*)}.
\end{multline*}
Further, for fixed $N_1$ and $I_{N_1}\in\mathcal J_{N_1}$ observe
\begin{align}\label{4::2}
&\sum_{N_2\le N_1}\sum_{\twoline{I\subset\mathcal J_{N_2}}{I\subset I_{N_1}}}\|b^{A,N_2}_{s_2,1}\|_{\ell^1(I^*)}=\notag\\
&\qquad=\sum_{N_2\le N_1}\sum_{\twoline{I\subset\mathcal J_{N_2}}{I\subset I_{N_1}}}
\sum_{x\in I^*}\Big|\sum_{Q\in\mathcal D_{A,N_2,s_2}^1}
\Big(1_Qf_0^{\frac{\lambda N_2}{A}}(x)-E_{Q}f_0^{\frac{\lambda N_2}{A}}(x)\Big)\Big|\notag\\
&\qquad\le 2\sum_{N2\le N_1}\sum_{\twoline{I\subset\mathcal J_{N_2}}{I\subset I_{N_1}}}
\sum_{x\in I^*}\sum_{Q\in\mathcal D_{A,N_2,s_2}^1}
1_Qf_0^{\frac{\lambda N_2}{A}}(x)\\
&\qquad\le c\sum_{\twoline{Q\in\mathcal D_{A,N_2,s_2}^1}{Q\subset I_{N_1}^*}}\sum_{x\in Q}1_Qf_0^{\frac{\lambda N_2}{A}}(x).\notag
\end{align}
We have used the fact, that scales of cubes in $\mathcal D_{A,N_2,s_2}^1$ are all smaller than $N_2^\alpha$, and also that for fixed $A,s_2$ the cubes in $\mathcal D_{A,N_2,s_2}^1$ are disjoint.
 Consequently, again using disjointness and \eqref{CZ_decom}
\begin{align}\label{4::3}
&\sum_{N_2\le N_1}\sum_{\twoline{Q\in\mathcal D_{A,N_2,s_2}^1}{Q\subset I_{N_1}^*}}\sum_{x\in Q}1_Qf_0^{\frac{\lambda N_2}{A}}(x)\le\\
&\qquad\qquad\le C\sum_{Q\subset I_{N_1}^*}\sum_{x\in Q}\Big(\sum_{N_2\le N_1} 1_{Q\in\mathcal D_{A,N_2,s_2}^1}(x)\,f_0^{\frac{\lambda N_2}{A}}(x)\Big)\notag\\
&\qquad\qquad\le C\sum_{Q\subset I_{N_1}^*}\sum_{x\in Q}\Big(\sum_{N_2\le N_1}f_0^{\frac{\lambda N_2}{A}}(x)\Big)\notag\\
&\qquad\qquad\le C\sum_{Q\subset I_{N_1}^*}\sum_{x\in Q}f_0(x)\notag\\
&\qquad\qquad\le C\sum_{Q\subset I^*_{N_1}}\lambda|Q|\notag\\
&\qquad\qquad\le C\lambda|I_{N_1}|\notag\\
&\qquad\qquad\le C\lambda N_1^\alpha \notag
\end{align}
Thus,
\begin{align*}
I&\le c\sum_{\twoline{s_i:2^{s_i}\ge A^{}}{i=1,2}}2^{-\delta(s_1+s_2)}\sum_{N_1}\sum_{I\in\mathcal J_{N_1}}
\frac{A^{3\epsilon}}{\lambda}\|b^{A,N_1}_{s_1,1}\|_{\ell^1(I^*)}\\
&\le c\sum_{N_1}\sum_{I\in\mathcal J_{N_1}}\frac{A^{3\epsilon}}{\lambda A^{{\delta}{}}}\|f_0^{\frac{\lambda N_1}{A}}\|_{\ell^1(I^*)}\\
&\le \frac{c}{\lambda A^{\frac{\delta}{2}}},
\end{align*}
using disjointness of the supports of $f_0^{\frac{\lambda N_1}{A}}$.

We now turn to $II$ in \eqref{4::1}, and proceed similarly.
\begin{align*}
II&=\sum_J\sum_{N}\sum_{s:2^s\ge A^{}}\frac{A^{3\epsilon}|J|}{\lambda^2 N^{\alpha+1}}\|b^{A,N}_{s,1}\|_{\ell^2(I_{N}^*(J))}^2\\
&\le c\sum_J\sum_{N}\sum_{s:2^s\ge A^{}}\frac{A^{3\epsilon}|J|}{\lambda^2 N^{\alpha+1}}\|f^{A,N}_{s,1}\|_{\ell^2(I_{N}^*(J))}^2\\
&= c\sum_J\sum_{N}\sum_{s:2^s\ge A^{}}\frac{A^{3\epsilon}|J|}{\lambda^2 N^{\alpha+1}}\sum_{\twoline{Q\subset I_{N}^*(J)}{|Q|\sim(2^{-s}N)^\alpha}}\|(f^{A,N}_{s,1})^2\|_{\ell^1(Q)}\\
&\le c\sum_J\sum_{N}\frac{A^{3\epsilon}|J|}{\lambda^2 N^{\alpha+1}}\sum_{Q\subset I_{N}^*(J)}\|(f^{A,N}_{s,1})^2\|_{\ell^1(Q)}\\
&\le c\sum_{N}\sum_{I_N}\frac{A^{3\epsilon}}{\lambda^2 N^{\alpha+1}}\sum_{Q\subset I_N^*}\|(f^{A,N}_{s,1})^2\|_{\ell^1(Q)}\sum_{J\subset I_N^*}|J|\\
&\le c\sum_{N}\frac{A^{3\epsilon}}{\lambda^2 N}\sum_{Q}\|(f^{A,N}_{s,1})^2\|_{\ell^1(Q)}\\
&\le c\sum_{N}\frac{A^{3\epsilon}}{\lambda A}\sum_{Q}\|f^{A,N}_{s,1}\|_{\ell^1(Q)}\\
&=c\frac{A^{3\epsilon}}{\lambda A}\sum_{N}\|f^{A,N}_{s,1}\|_{\ell^1}\\
&\le c\frac{\|f_0\|_{\ell^1}}{\lambda A^{1-3\epsilon}}\\
&=c\frac{1}{\lambda A^{1-3\epsilon}}.
\end{align*}
We finally consider the last summand in \eqref{4::1}. We use the following two properties of the kernel $Err_{N,J}$ (Lemma \ref{k_reg}): its support (which is within $c N^\alpha$ from the center of $J$ in both variables and, for any $x,y$, and also in the strip $|x-y|\le CM$)
\begin{equation*}
\sum_J|Err_{N,J}(x,y)|\le \frac{c}{N^\alpha}.
\end{equation*}
which, by \eqref{CZ_decom}, leads to
\begin{equation*}
\sum_J\sum_{y}|Err_{N,J}(x,y)||b^{A,N}_{s_2,1}(y)|\le \frac{c\lambda (N+(2^{-s}N)^\alpha)}{N^\alpha}.
\end{equation*}

We have
\begin{align*}
III&=\sum_J\sum_{\twoline{s_i:2^{s_i}\ge A^{}}{i=1,2}}\sum_{N}\frac{A^{3\epsilon}}{\lambda^2}\big|\langle Err_{N,J}b^{A,N}_{s_1,1},b^{A,N}_{s_2,1}\rangle\big|\\
&\le C\,\frac{A^{3\epsilon}}{\lambda^2}\sum_{\twoline{s_i:2^{s_i}\ge A^{}}{i=1,2}}\sum_{N}\sum_J\sum_{x,y\in I^*_N(J)} |Err_{N,J}(x,y)|\,|b^{A,N}_{s_1,1}(x)|\,|b^{A,N}_{s_2,1}(y)|\\
&\le C\,\frac{A^{3\epsilon}}{\lambda^2}\sum_{\twoline{s_i:2^{s_i}\ge A^{},s_1\ge s_2}{i=1,2}}\sum_{N}
\sum_{I\in\mathcal J_N}\frac{\lambda\,(N+(2^{-{s_1}}N)^\alpha)}{N^\alpha}
\|b^{A,N}_{s_2,1}\|_{\ell^1(I^*)}\\
&\le \frac{C\|f_0\|_{\ell^1}}{\lambda}.
\end{align*}

\noindent{\bf Case $i=2$.} Recall \eqref{D1_def}
\begin{equation*}
   \mathcal D_{A,N,s}^2=\{Q: |Q|\sim(2^{-s}N)^\alpha,\,2^s\le A\},
\end{equation*}
and \eqref{bsi_def}
\begin{equation*}
b^{A,N}_{s,2}(x)=\sum_{Q\in \mathcal D_{A,N,s}^2}
\Big(1_{Q}(x)f^{\frac{\lambda N}{A}}(x)-E_{Q}f^{\frac{\lambda N}{A}}(x)\Big).
\end{equation*}
We have
\begin{multline*}
  \max_B\Big|\sum_s\sum_{\twoline{A}{A^{}\ge 2^s}}\sum_{N\le B}\mu_N*b_{s,2}^{A,N}(x)\Big|^2\le\\
 \le  c_\epsilon\sum_s2^{\epsilon s}\max_B\Big|\sum_{A:\,A^{}\ge 2^{s}}\sum_{N\le B}\mu_N*b_{s,2}^{A,N}(x)\Big|^2.
\end{multline*}
Hence, for some $s$, we must have
\begin{equation*}
  \max_B\Big|\sum_{N\le B}\mu_N*b_{s,2}^{N}(x)\Big|^2>2^{-2\epsilon s}\lambda_0^2,
\end{equation*}
where
\begin{equation*}
  b_{s,2}^N=\sum_{A:\,A^{}\ge 2^{s}}b_{s,2}^{A,N}.
\end{equation*}
Now, apply lemmas \ref{max_sparse_est}, \ref{max_cs} \ref{jmax_est} exactly as in the case $i=1$. The Theorem follows. \qed

\noindent{\bf Remark.} If $|Q|\le M^{\theta-2\epsilon}$ then   $A\ge M^{\epsilon}$ by Lemma \ref{key_CZ}.
By Lemma \ref{sqr_UZ} or by  \cite{UZ} we have
\begin{equation}
\sum_N\sum_{I_N} |\mu_N*b^{A,N}_{Q}(x)|^2\le \frac{C\lambda}{M^\frac{\delta\epsilon}2}
\end{equation}
This yields the desired maximal estimate since we have at most $C\log M$ summands. We omit the details.

\begin{proof}[proof of Corollary \ref{corollary}]
Suppose $E\subset\zet$ is finite, and $f=\frac{{\mathbf{1}}_{E}}{|E|}$. Fix $\lambda>0$, and let $N\in\en$ be such that
\begin{equation*}
  \lambda N\le \frac{1}{|E|}\le2\lambda N,\qquad \text{where }f(x)=\frac{1}{|E|}\neq0,\quad\text{for $x\in E$}.
\end{equation*}
We thus have $\lambda N\sim1/|E|$. We decompose the maximal Hilbert transform:
\begin{align*}
  \ha^* f&\le\max_{A}\Big|\sum_{A\le M<N^\theta}\mu_M*f\Big|+\max_{A}\Big|\sum_{\twoline{N^\theta \le M<N^D}{A\le M}} \mu_M*f\Big|\\
        &\qquad+\max_{A\ge N^D}\Big|\sum_{A\le M}\mu_M*f\Big| \\
  &= I+II+III,
\end{align*}
where $D$ is an appropriate constant to be determined later.For fixed $A,\ N$ the summation index sets can be empty We estimate the superlevel sets of each of the components.
\newline
I: We have
\begin{equation*}
  \left|\bigcup_{M<N^\theta}\rm{supp}\, (\mu_M*f)\right|\le CN^\theta|E|\le C \frac{N^\theta}{\lambda N}= C\frac{N^{\theta-1}}{\lambda}\le C\frac{\|f\|_{\ell^1}}{\lambda},
\end{equation*}
(recall, that $\theta<1$).
\newline
II: This case is a consequence of the main Theorem \ref{theorem}
\newline
III:  For a fixed dyadic $M$ we further decompose:
\begin{equation*}
  \mu_M = \mu_M*\varphi_{M^\rho}+(\mu_M-\mu_M*\varphi_{M^\rho})=K_M+T_M,
\end{equation*}
where $\rho>\alpha(\alpha-1)<\rho<\alpha$ is arbitrary and fixed. From Lemma \ref{k_reg} we know, that
\begin{equation*}
  \|T_M\|_{\ell^2\to\ell^2}\le M^{-\beta},
\end{equation*}
where $\beta>0$ depends only on $\rho$ and $\alpha$. Thus
\begin{equation*}
  \bigg\|\sum_{N^D\le A\le M}T_M\bigg\|_{\ell^2\to\ell^2}\le C A^{-\beta}\le CN^{-\beta D},
\end{equation*}
and so
\begin{align*}
&\bigg|\Big\{x:\max_{A\ge N^D}\big|\sum_{A\le M}T_M*f\big|\ge\lambda\Big\}\bigg|\le\\
  &\qquad\le\sum_{\twoline{A-\text{dyadic}}{A\ge N^D}}\bigg|\Big\{x:\sum_{A\le M}T_M*f>\lambda\Big\}\bigg|\\
  &\qquad\le\sum_{\twoline{A-\text{dyadic}}{A\ge N^D}}\frac{1}{\lambda^2}\cdot\frac{1}{A^{2\beta}}\cdot\|f\|_{\ell^2}^2\\
  &\qquad\le \sum_{\twoline{A-\text{dyadic}}{A\ge N^D}}\frac{1}{\lambda^2 \cdot A^{2\beta}}\cdot 2\lambda N\cdot\|f\|_{\ell^1}\\
  &\qquad=\frac{1}{\lambda}\cdot N^{1-2\beta D}\cdot\|f\|_{\ell^1}\\
  &\qquad\le\frac{1}{\lambda}\|f\|_{\ell^2},
\end{align*}
where we specify $D$ so that $2\beta D\ge1$. We are left with an estimate of
\begin{equation*}
  \max_{A\ge N^D}\Bigg|\sum_{A\le M} K_M*f\Bigg|.
\end{equation*}
Using Lemma 3.5 from \cite{PZ1}
 this is a (maximal) Calderon-Zygmund operator, and in particular of weak type $(1,1)$.
\end{proof}

\end{document}